\documentclass[reqno,10pt]{amsart}

\usepackage{amsfonts,amssymb,latexsym,color,amsmath,enumerate,stmaryrd,amsthm,ushort,dsfont,mathrsfs}

\usepackage{hyperref} %pacote para linkagem
\usepackage{graphicx}

\usepackage[numbers]{natbib}
\newtheorem{theorem}{Theorem}[section]             \newtheorem{proposition}[theorem]{Proposition}     \newtheorem{lemma}[theorem]{Lemma}

\newtheorem{definition}[theorem]{Definition}
\newtheorem{remark}{Remark}[section]

\newcommand{\1}{\ensuremath{\mathds{1}}}

\makeatletter
\@addtoreset{equation}{section}
\makeatother

\renewcommand\thetable{\thesection.\@arabic\c@table}

\title[Fluctuations of the occupation density for a parking process]{Fluctuations of the occupation density for a parking process}

\author{Cristian F. Coletti, Sandro Gallo, Alejandro Rold\'an-Correa and  Le\'on A. Valencia}

\date{}

\address{
\newline
Cristian F. Coletti
\newline
Centro de Matem\'atica, Computa\c{c}\~ao e Cogni\c{c}\~ao,  Universidade Federal do ABC.
\newline
Santo Andre, S\~ao Paulo, Brazil.
\newline
e-mail:  cristian.coletti@ufabc.edu.br
\newline
\newline
Sandro Gallo
\newline
Departamento de Estat\'istica, Universidade Federal de S\~ao Carlos.
\newline
S\~ao Carlos, S\~ao Paulo, Brazil.
\newline
e-mail: sandro.gallo@ufscar.br
\newline
\newline
Alejandro Rold\'an
\newline
Instituto de Matem\'aticas, Universidad de Antioquia.
\newline
Medell\'in, Antioquia, Colombia.
\newline
e-mail:  alejandro.roldan@udea.edu.co
\newline
\newline
Le\'on A. Valencia
\newline
Instituto de Matem\'aticas, Universidad de Antioquia.
\newline
Medell\'in, Antioquia, Colombia.
\newline
e-mail:  lalexander.valencia@udea.edu.co.
}

\thanks{Research funded by FAPESP (2022/08948-2) and Universidad de Antioquia (2023-58830).}

\begin{document}

\begin{abstract} Consider the following simple parking process on $\Lambda_n := \{-n, \ldots, n\}^d,d\ge1$: at each step, a site $i$ is chosen at random in $\Lambda_n$ and if $i$ and all its nearest neighbor sites are empty, $i$ is occupied. Once occupied, a site remains so forever. The process continues until all sites in $\Lambda_n$ are either occupied or have at least one of their nearest neighbors occupied. The final configuration (occupancy) of $\Lambda_n$ is called the jamming limit and is denoted by $X_{\Lambda_n}$. Ritchie~\cite{MR2205908} constructed a stationary random field on $\mathbb Z^d$ obtained as a (thermodynamic) limit of the $X_{\Lambda_n}$'s as $n$ tends to infinity.
As a consequence of his construction, he proved a strong law of large numbers for the proportion of occupied sites in the box $\Lambda_n$ for the random field $X$. Here we prove the central limit theorem, the law of iterated logarithm, and a gaussian concentration inequality for the same statistics. A particular attention will be given to the case $d=1$, in which we also obtain new asymptotic properties for the sequence $X_{\Lambda_n},n\ge1$ as well as a new proof to the closed-form formula for the occupation density of the parking process. 
\end{abstract}

\maketitle

\section{Introduction}

\label{sec:intro}

Parking processes, also known as \textit{random sequential adsorption} models, refer to processes in which particles are randomly introduced {into a bounded domain, and are adsorbed (and remain as such forever) only if they do not overlap with any previously adsorbed particle.} A question of interest in these models is to determine the proportion of the space that is covered by the adsorbed particles at the end of the process (i.e., when it is not possible to absorb more particles). Variations and generalizations of random sequential adsorption models have received attention for their applications in the physical and chemical sciences. For an overview and motivation, see, for instance, \cite{RevModPhys.65.1281, 10.2307/2983806}. Many analytical results have been established in one-dimensional models (see \cite{MR1282578, doi:10.1063/1.1733154, 10.2307/2983806, renyi1958one}) and also in quasi-one-dimensional models (see \cite{MR1162872,MR1282579}). In higher dimensions, {some statistical properties} have been studied via computer simulations (see \cite{RevModPhys.65.1281, TALBOT2000287}). However, there exist few rigorous theoretical results. In this regard, we highlight the contributions of Penrose~\cite{MR1887532} and Ritchie~\cite{MR2205908} who proved asymptotic results for {special classes} of these models.

\vspace{0,2cm}

In this work, we consider a version of a parking process on the box $\Lambda_n := \{-n, \ldots, n\}^{d},d\ge1$ with a simple occupation scheme.
 It is constructed iteratively as follows.
\begin{itemize}
\item[1:] Start with all sites in $\Lambda_n$ empty.
\item[2:] A non yet chosen site in $\Lambda_n$ is chosen at random, and if all its nearest {neighbors} are empty, then the site is declared occupied. Once occupied, a site remains so forever.
\item[3:] Repeat step 2 until all sites in $\Lambda_n$ have been chosen (that is, are either occupied or have at least one of their nearest neighbors occupied).
\end{itemize}
Assigning $1$'s to occupied sites and $0$ to non-occupied sites, we end up with a configuration in $\{0,1\}^{\Lambda_n}$, called the \textit{jamming limit of $\Lambda_n$}, and  denoted by $X_{\Lambda_n}$.

Penrose~\cite{MR1887532}  proved the  law of large numbers ($L^p$-convergence) and  the Central Limit Theorem (CLT) for $\frac{1}{|\Lambda_n|}\sum_{i\in \Lambda_n}X_{\Lambda_n}(i),n\geq0$, the sequence of occupancy densities of the jamming limits.

Later on,  Ritchie~\cite{MR2205908} proved that there exists a translation invariant random field $X$ on $\mathbb Z^d$ such that $X_{\Lambda_n}\rightarrow X$ almost surely. This $X$ can be seen as the jamming limit on $\mathbb Z^d$ of the parking process, and in the statistical physics terminology, it is the thermodynamic limit of the jamming limits $X_{\Lambda_n},n\ge1$.  His proof is constructive: he designed an algorithm that samples from  $X$ in any finite region $\Lambda\subset\mathbb Z^d$. As a corollary of his construction, he  obtained the strong law of large numbers for the sequence of occupancy densities of the jamming limits. 

\vspace{0,2cm}

In the present paper, we take advantage on the construction of Ritchie~\cite{MR2205908} to obtain new results for the proportion of occupied sites in $\Lambda_n$ for the parking process. We can split the main contributions of the present paper into two parts: results that hold for any $d\ge1$, and results that we were able to prove for the specific  case $d=1$. 

For any $d\ge1$, we prove the central limit theorem (CLT) and the law of iterated logarithm (LIL) for $\frac{1}{|\Lambda_n|}\sum_{i\in\Lambda_n}X(i),n\ge1$, the sequence of proportion of occupied sites in $\Lambda_n$ in the thermodynamic jamming limit. We also provide a gaussian concentration inequality (GCI), that quantifies precisely how the proportion of occupied sites deviates from its mean, in any finite boxes $\Lambda_n$ (non-asymptotic result). 
The proofs of these results use the construction of the thermodynamic limit $X$ provided by Ritchie~\cite{MR2205908}. Indeed, his algorithm allows us to show good mixing properties of the random field $X$, and then, to rely on existing results for stationary mixing random fields. 
These results for $d\ge1$ should be contrasted with those of Penrose~\cite{MR1887532}: his convergence results are obtained for the sequence free boundary boxes $X_{\Lambda_n},n\ge1$ while we obtain our results directly on the random field $X$. In other words, we consider the proportion of occupied sites $\frac{1}{|\Lambda_n|}\sum_{i\in\Lambda_n}X(i),n\ge1$ and he considers $\frac{1}{|\Lambda_n|}\sum_{i\in \Lambda_n}X_{\Lambda_n}(i),n\geq1$.

 The above observation motivated us to present other results for the case $d=1$. Indeed, under this restriction, we are able to transpose the above results to the sequence of proportions $\frac{1}{|\Lambda_n|}\sum_{i\in\Lambda_n}X_{\Lambda_n}(i),n\ge1$. So in particular, the LIL and GCI are new results in this context. Finally, we establish an alternative proof of the closed-form formula $\frac{1}{2}(1-e^{-2})$ for the occupation density of the one-dimensional parking process. This calculation, different to a previously existing argument of the literature (see Fan and Percus~\cite{MR1149487}), is based on a direct calculation of probability that the origin be occupied in $X$, and does not involve any limiting procedure. 

\vspace{0,2cm}
The paper is organized as follows. In Section 2, we introduce some notation, explain the construction of Ritchie~\cite{MR2205908} and state our main results. The proofs are given in  Section~3.

\section{Definitions and notation}

\subsection{Notation}
Throughout this paper we use  the following notation. Let ${\bf0}:=(0,0)$ denote the origin of $\mathbb{Z}^d$. For $i\in\mathbb{Z}^d,$ let $\|i\|$ denote the Euclidean norm of $i$. Given two sites $i,j\in \mathbb{Z}^d$, we say that they are nearest neighbors if $\|i-j\|=1$,  denoted by $i \sim j$. For $A\subset\mathbb{Z}^d,$ we denote by $|A|$ the number of elements of $A$ and by $\partial A$  the boundary of $A$, that is, $\partial A=\{i\in A: \text{there exists } j\notin A \text{ and } i\sim j\}$.  Notation $\Lambda_n(i):=\{j\in\mathbb Z^d:||i-j||_{\max}\le n\}$ stands for a box of size $n$ centered on $i\in\mathbb Z^d$, and in particular $\Lambda_n:=\Lambda_n(0)$.

For any set $\Lambda\subset \mathbb Z^d$ and any collection of random variables $X(i),i\in\mathbb Z^d$, we write $X(\Lambda)$ to denote the vector $\{X(i)\}_{i\in\Lambda}$. We use the shorthand notation $X=\{X(i)\}_{i\in\mathbb Z^d}$ for the  random field.

For any \(\Lambda\subset \mathbb{Z}^d\), we denote by \(\mathcal{F}_\Lambda\)  the \(\sigma\)-algebra generated by the random variables \(X(i)\), \(i\in \Lambda\).

\vspace{0.2cm}

\subsection{Construction of the jamming limit of the parking process on $\mathbb{Z}^d$} Recall the construction of the process in $\Lambda_n$ given in introduction. We start giving an alternative construction, using i.i.d. random variables in $(0,1)$ instead of iteratively choosing sites uniformly at random as done in introduction. We also extend the construction to any bounded region $\Lambda$ and any boundary configuration.

\begin{definition}[The \textit{parking process} on $\Lambda\subset \mathbb Z^d$ with boundary condition $x\in\{0,1\}^{\mathbb Z^d}$]\label{P.P.finite}
Fix  a configuration $x$ and let $U=\{U(i)\}_{i\in\mathbb{Z}^d}$ be a family of independent and identically distributed (i.i.d) random variables with a uniform distribution on $(0,1)$.
\begin{itemize}
\item[1:] Set $X^{(x)}_\Lambda(i)=0$ for $i\in \Lambda$ and $X^{(x)}_\Lambda(i)=x(i)$ for $i\in \Lambda^c$;
\item[2:] choose $i\in \Lambda$ such that $U(i)=\min\{  U(j): j\in \Lambda \text{ and } j \text{ has not been chosen previously}\}$;
\item[3:] if $X^{(x)}_\Lambda(j)=0$ for all $j\sim i$, then set $X^{(x)}_\Lambda(i)=1$. Otherwise, $X^{(x)}_\Lambda(i)=0$;
\item[4:] if there are points in $\Lambda$ not chosen yet, then go back to step $2$. Otherwise, stop the algorithm.
\end{itemize}
We call the final configuration $X^{(x)}_\Lambda$ the jamming limit of the parking process on $\Lambda$ with boundary condition $x$.
\end{definition}

Observe that the jamming limit $X^{(x)}_\Lambda$ on a finite subset $\Lambda$ of $\mathbb{Z}^d$ is constructed as a random element in $\{0,1\}^{\mathbb{Z}^d}$, that is, $X^{(x)}_\Lambda$ is a random field with frozen configuration $x$ on $\Lambda^c$. The projection on $\Lambda_n$ of the random field $X^{({0})}_{\Lambda_n}$ (notation for $X^{(x)}_{\Lambda_n}$ when $x(i)=0,i\in\mathbb Z^d$) has the same distribution as the vector  constructed in introduction. 

 In order to formalize the construction of the jamming limit of the parking process on the whole grid $\mathbb{Z}^d$, we first need to define the concept of \emph{armour}.

\begin{definition}[The armour of $\Lambda\subset\mathbb{Z}^d$ with respect to $U$] For $i,j\in\mathbb{Z}^d$, we write  $U(i\downarrow j)$ if there exist $i_0, \ldots, i_n\in \mathbb{Z}^d$ with $i_0=i$, $i_n=j,$ $\|i_{k}-i_{k-1}\|=1,$ $k=1,\ldots,n$, and the event $\{U(i_0)>U(i_1)>\ldots >U(i_n)\}$ occurs. Besides, given a subset $\Lambda\subset \mathbb{Z}^d$, we define the \textit{armour} of $\Lambda$ as the random subset
\begin{equation*}
\mathcal{A}(\Lambda):=\Lambda \cup \{ j\in \mathbb{Z}^d: \text{there exists } i \in \Lambda \text{ such that } U(i\downarrow j) \}.
\end{equation*}
\end{definition}
Observe that formally, the armour is obtained as a function of the random field $U$. When necessary to avoid possible confusions, we will emphasize this dependence using the notation $\mathcal A(\Lambda)(U)$.

For any finite set $\Lambda\subset \mathbb{Z}^d$, $\mathcal{A}(\Lambda)$ is almost surely finite, see \cite[Lemma 31]{MR2205908}. This allows defining, for each $i\in\mathbb{Z}^d,$ the jamming limit on the armour $\mathcal{A}(\{i\})$ according to Definition~\ref{P.P.finite}.
With this, Ritchie~\cite{MR2205908} constructed the {random field $X$ as a deterministic function of the random field $U$ through
\begin{equation*}
X(i):=X^{({0})}_{\mathcal{A}(\{i\})}(i),\hspace{0.5cm} \forall i\in\mathbb{Z}^d.
\end{equation*}
Here also, the field $X$ is obtained as a deterministic function of the field $U$. When necessary to avoid possible confusions, we will emphasize this dependence using the notation $X(U)$.
As a direct consequence of this construction he also proved that}
\begin{equation*}
\lim_{n\to\infty}X^{({0})}_{\Lambda_n}=X, \hspace{0.5cm} \text{almost surely.}
\end{equation*}
The random field $X$ is called the \textit{thermodynamic jamming limit} of the parking process on $\mathbb{Z}^d.$

In order to make the distinction later on, let us introduce the following notation for proportions of occupied sites, depending on whether we refer to the random field $X$ or to the free boundary sequence of vectors $X^{(0)}_{\Lambda_n},n\ge1$:
\[
\rho_n:=\frac{N_n}{|\Lambda_n|}:=\frac{\sum_{i\in \Lambda_n}X(i)}{|\Lambda_n|}\,\,\,\text{and }\,\,\,\bar\rho_n:=\frac{\bar N_n}{|\Lambda_n|}:=\frac{\sum_{i\in \Lambda_n}X^{(0)}_{\Lambda_n}(i)}{|\Lambda_n|}
\]

\vspace{0.2cm}

\subsection{Main results for any $d\geq1$}
Here we state the results we have proved in any dimension. We start with asymptotic properties.  The strong law of large numbers for the \textit{occupation density} for the thermodynamic jamming limit
\begin{equation}\label{rho}
\lim_{n\to \infty}\rho_n=\mathbb{E}[X({\bf0})]=:\rho,
\hspace{0.5cm} \text{almost surely,}
\end{equation}
was  proved by  Ritchie~\cite[Theorem 42]{MR2205908}. 
In the next result, we provide the Central Limit Theorem and a Law of Iterated Logarithm for the number of occupied sites in the box $\Lambda_n$ by the thermodynamic jamming limit on $\mathbb{Z}^d$.

\begin{theorem}\label{main}
Let $N_n=\sum_{i\in \Lambda_n}X(i)$ be the number of occupied sites in $\Lambda_n$ (relative to the thermodynamic jamming limit $X$). It satisfies the central limit theorem,
$$
\frac{N_n-|\Lambda_n|\rho}{\sqrt{\mathrm{Var}(N_n)}}\overset{\mathcal{D}}{\underset{n\to\infty}{\longrightarrow}} N(0,1),
$$
where $\rho$ is defined in (\ref{rho}) and the law of iterated logarithm
\[
\mathbb P\left(\limsup_n\frac{N_n-|\Lambda_n|\rho}{\sqrt{2\mathrm{Var}(N_n)\log\log |\Lambda_n|}}=1\right)=1.  
\]
\end{theorem}

Next we state a result quantifying precisely the probability that the proportion deviates from its mean in finite boxes $\Lambda_n$.

\begin{proposition}\label{prop:concentration}
For any $\epsilon>0,n,d\ge1$
\begin{equation}\label{eq:concentrationmain}
\mathbb P\left(\left|{N_n}-\rho{|\Lambda_n|}\right|>\epsilon\right)\le  e^{\frac{1}{e}-\frac{\epsilon^2}{4eB|\Lambda_n|}}
\end{equation}
where 
\[
B=1+\frac{2d}{2d-1}\sum_{k\ge1}\frac{(2d-1)^{k}[(2k+1)^d-(2k-1)^d]}{k!}.
\]
\end{proposition}

Such inequalities are reminiscent of the so-called Azuma inequalities for martingales. It is called sub-gaussian, because the upper bound essentially recovers the optimal exponential exponent $\epsilon^2/|\Lambda_n|$ we would get if we were quantifying the concentration of a sum of i.i.d. gaussian random variables around its mean. The constant $B$ is explicit although not trivial to get for general $d$, but note that for $d=1$, it simplifies to $B=4e-3$.

\vspace{0,2cm}

The above results hold for any dimension but are restricted to the study of the random field $X$. With respect to the sequence of free boundary fields $X_{\Lambda_n}^{(0)},n\ge1$,  Ritchie~\cite[Theorem 42]{MR2205908} proved a strong law of large numbers for $\frac{\bar N_n}{|\Lambda_n|}$ (to the same limiting value $\rho$) and  Penrose~\cite{MR1887532} proved a Central Limit Theorem. 
%This motivated us to try to transpose the results we obtained for $X$ to results for $X_{\Lambda_n}^{(0)},n\ge1$. Informally, a natural path to do this is to show that the boundary effect $|N_n-\bar N_n|$ gets negligible in probability. We have been able to use this path only for $d=1$ so far, for this reason we decided to dedicate a separate subsection with specific results to this case.
Penrose~\cite{MR1887532} also proved the convergence of $\frac{\bar N_n}{|\Lambda_n|},n\ge1$ in $L^p$ for any $p$ but did not provide bounds. The next result quantifies the deviation in mean  for the numbers of occupation $\bar N_n,n\ge1$.

\begin{proposition}\label{prop:mean-dev}
\begin{equation*}
\left|\,\mathbb E \bar N_n-|\Lambda_n|\rho\,\right|\le \left\{\begin{array}{ccc}
2(e-1),&\text{if $d=1$,}\\
\frac{2d(2d-1)^{n}}{(n+1)!}+(2d)^2\sum_{k=0}^{n-1}\frac{(2d-1)^{k}(2(n-k)+1)^{d-1}}{(k+1)!},&\text{if $d\ge1$.}
\end{array}\right.\end{equation*}
\end{proposition}
Here also, the upper bound is not easy to get in general dimensions but gives the upper bound to $4(e-1)$ for $d=1$. Notice that Gerin~\cite{gerin2015page}  obtained a bound for $d=1$  which is much less precise than our. 
\vspace{0.2cm}

\subsection{Results specific to $d=1$}
Even though the study of the sequence $X_{\Lambda_n}^{(0)},n\ge1$ in $d=1$ has an extensive literature, most works are simulation based and few works obtain rigorous results on the statistical properties of the model. 
It seems that the rigorous calculation of $\rho=\frac{1}{2}(1-e^{-2})$ was done in the 40's as the limiting average proportion of occupied sites in the jamming limit, using the combinatorial techniques of Flory \cite{flory1939intramolecular}
and has been re-discovered several times, see for instance Fan and Percus \cite{MR1149487} (see also \cite{RevModPhys.65.1281} for a nice overview up to the 90's). 
It is important to reforce that at that time, $\rho$ was not proved to be the limit proportion of occupied sites still. 
Page \cite{10.2307/2983806} would prove the convergence of the sample mean to $\rho$ in probability but, 
as far as we know,  it is Ritchie~\cite{MR2205908} who was the first to actually prove the  almost sure convergence (his results hold in any dimension). To conclude, let us mention that the problem is considered in the literature under different perspectives, such as seating arrangement/fatmen/packing problem or even independents sets in graph theory   (see for instance \cite{flajolet, friedman1964solution, gerin2015page, pinsky2014one}).

Although the value of $\rho=\frac{1}{2}(1-e^{-2})$ is well-known, we found it interesting to mention that Ritchie's construction allows one more alternative calculation of this constant, using directly (\ref{rho}) without any limiting procedure, similar to the approach of \cite{gerin2015page}. 

\begin{proposition}\label{density1}
Let \(X:\mathbb{Z}\to\{0,1\}\) be the thermodynamic limit of the parking process on \(\mathbb{Z}\). Then,
\[\rho:=\mathbb{E}[X({0})]=\frac{1}{2}(1-e^{-2}).\]
\end{proposition}

The next result is a consequence of Theorem \ref{main} and Proposition \ref{prop:concentration} when we consider $\bar N_n,n\ge1$ instead  of $N_n,n\ge1$. As far as we know, the LIL and the concentration inequality are new results. 

\begin{proposition}\label{prop:d=1}
In $d=1$, the sequence $\bar N_n=\sum_{i\in \Lambda_n}X_{\Lambda_n}(i)$  satisfies the central limit theorem
\[
\frac{\bar N_n-\mathbb E(\bar N_n)}{\sqrt{|\Lambda_n|\sigma^2}}\overset{\mathcal{D}}{\underset{n\to\infty}{\longrightarrow}} N(0,1)
\]
and the law of iterated logarithm
\[
\mathbb P\left(\limsup_n\frac{\bar N_n-\mathbb E(\bar N_n)}{\sqrt{2|\Lambda_n|\sigma^2\log\log |\Lambda_n|}}=1\right)=1
\]
for some constant $\sigma^2\in(0,\infty)$.  
Moreover, for any $\epsilon>0,n\ge1$
\begin{equation}\label{eq:concentration}
\mathbb P\left(\left|{\bar N_n}-\rho|\Lambda_n|\right|>\epsilon\right)\le e^{\frac{1}{e}-\frac{\epsilon^2}{16e(4e-3)|\Lambda_n|}}+2\frac{1}{\lceil \epsilon/4+2\rceil!}.
\end{equation}

\end{proposition}

The last statement can be used to consider the concentration around the mean $\mathbb E(\bar N_n)$. Indeed
\[
\mathbb P\left(\left|{\bar N_n}-\mathbb E(\bar N_n)\right|>\epsilon\right)\le \mathbb P\left(\left|{\bar N_n}-\rho|\Lambda_n|\right|>\epsilon-\left|\mathbb E(\bar N_n)-\rho|\Lambda_n|\right|\right)\le \mathbb P\left(\left|{\bar N_n}-\rho|\Lambda_n|\right|>\epsilon-2(e-1)\right)
\]
(using  by Proposition \ref{prop:mean-dev} for the second inequality) which can be bounded using \eqref{eq:concentration}. 

Notice also that for small $\epsilon$, \eqref{eq:concentration} is not quite a subgaussian inequality,  due to boundary effects. However, for large $\epsilon$, the second term may be smaller than the first one, thus yielding a subgaussian inequality. For instance, if we take $\epsilon=|\Lambda_n|$ we get that the first term is $e^{\frac{1}{e}-\frac{|\Lambda_n|}{16e(4e-3)}}$ while the second term is $2\frac{1}{\lceil |\Lambda_n|/4+2\rceil!}=o(e^{-|\Lambda_n|})$. 
%%%%%%%%%%%%%%%%%%%%%%%%%%%%%%%%%%%%%%%%%%%%%%%%%%%%%%%%%%%%%%%%%%%%%%%%%%%%%%%%%%%%%%%%%%%%%%%%

\section{Proofs}
Subsections \ref{sec:aux_results}, \ref{sec:techlemmas} and \ref{proof_mainthm} are entirely devoted to the proof of Theorem~\ref{main}, and are presented in the  case $d=2$ in order to simplify the presentation, and because the dimension does not have importance in the precision of the results. On the other hand, the proofs of Propositions \ref{prop:concentration} and \ref{prop:mean-dev}, given in Subsections   \ref{sec:proof_conc} and \ref{sec:proof_expect} respectively, are done for any $d\ge1$, because the constants involved in the statements are dimension dependents. In particular, as we said, the case $d=1$ plays an important role. Finally, Propositions~\ref{density1} and \ref{prop:d=1}, which holds only for $d=1$, is given in Subsection~\ref{sec:proof_prop}.

\subsection{Auxiliary definitions and results} \label{sec:aux_results}
The strategy to prove Theorem~\ref{main} consists of verifying the sufficient conditions given classical results of the literature for both,  the central limit theorem and the law of iterated logarithm to hold for general stationary mixing random field. To describe these conditions, we need to introduce a piece of notation (recalling that we present the proof of Theorem~\ref{main} in the case $d=2$ for ease of presentation). 

\begin{definition}
 For \(G\subset \mathbb{Z}^2\), let \(\mathcal{F}_G\) be the \(\sigma\)-algebra generated by the random variables \(X(i)\), \(i\in G\). If \(G,H\subset \mathbb{Z}^2\), let
\begin{equation*}%\label{eq:defd}
d(G,H):=\inf\{ \|i-j\|_{\max}:i\in G,j\in H \}.
\end{equation*}
If \(n\in\mathbb{N}\) and \(k,l\in\mathbb{N}\cup \{ \infty\}\), we define the mixing coefficient
\[\alpha_{k,l}(n):=\sup\{ |\mathbb{P}(A \cap B)-\mathbb{P}(A)\mathbb{P}(B)|\},\]
where the supremum is taken over all \(A\in\mathcal{F}_{G}\), \(B\in\mathcal{F}_{H}\), \(G,H\subset \mathbb{Z}^2\), with \(|G|\leq k\), \(|H|\leq l\) and \(d(G,H)\geq n\).
\end{definition}

Bolthausen's theorem \cite{MR672305} establishes the following criteria for a stationary random field to satisfy a central limit theorem.
\begin{theorem}[Bolthausen~\cite{MR672305}]\label{ThmMain}
    Let $X$ be a stationary random field on $\mathbb{Z}^2$. Assume that  
    \begin{enumerate}[(a)]    
        \item $\sum_{n=1}^\infty n\alpha_{k,l}(n)<\infty$ for $k+l\leq 4$.
        \item For some $\delta>0$, $\|X({\bf0}) \|_{2+\delta}<\infty$ and         $\sum_{n=1}^{\infty}n \alpha_{1,1}(n)^{\delta/(2+\delta)}<\infty$.
        \item $\alpha_{1,\infty}(n)=o(n^{-2})$.        
    \end{enumerate}
    Then
    \begin{equation}\label{mainresult1}
        \sum_{i\in\mathbb{Z}^2}|\mathrm{Cov}(X({\bf0}),X(i))|<\infty.
    \end{equation}
    If in addition  $\sigma^2=\sum_{i\in\mathbb{Z}^2}\mathrm{Cov}(X({\bf 0}),X(i))>0$, then
    \begin{equation*}%\label{mainresult2}
        \frac{N_n-|\Lambda_n|\rho}{\sigma |\Lambda_n|^{1/2}}\overset{\mathcal{D}}{\underset{n\to\infty}{\longrightarrow}}N(0,1).
    \end{equation*}
\end{theorem}

For the law of iterated logarithm, we will rely on the following result of \cite{nahapetian2013limit} (see Theorem 7.4.2 therein).

\begin{theorem}[Nahapetian~\cite{nahapetian2013limit}]\label{nahap}
    Let $X$ be a stationary random field on $\mathbb{Z}^2$. Assume that
   
    \begin{enumerate}[(a)]
       
        \item $\mathbb E |X|^{2+\delta}<\infty$ for some $\delta>0$,
       
        \item $\alpha_{k,l}(n)\le k^{\tau_1}l^{\tau_2}\alpha(n)$ for some constants $\tau_1,\tau_2$ and vanishing sequence $\alpha (n)$.
       
        \item The sequence $\alpha(n),n\ge1$ satisfies:
        \begin{itemize}
        \item for some $\delta',0<\delta'<\delta$, the series $\sum_{n\ge1}n\alpha(n)^{\frac{\delta'}{2+\delta'}}<\infty$,
        \item $\alpha(n)=O(\frac{1}{n^{2\beta }})$ where $\beta>2(\tau_1+\tau_2)+\frac{1}{2}$.
        \end{itemize}

    \end{enumerate}
    Then
    \begin{equation}\label{mainresult12}
        \sum_{i\in\mathbb{Z}^2}|\mathrm{Cov}(X({\bf0}),X(i))|<\infty.
    \end{equation}
    If in addition  $\sigma^2=\sum_{i\in\mathbb{Z}^2}\mathrm{Cov}(X({\bf 0}),X(i))>0$ and $\mathrm{Var}(N_n)=\sigma^2|\Lambda_n|(1+o(1))$ then
    \begin{equation*}%\label{nahap_iterated}
\mathbb P\left(\limsup_n\frac{N_n-|\Lambda_n|\rho}{\sqrt{2\sigma^2|\Lambda_n|\log\log |\Lambda_n|}}=1\right)=1.  
  \end{equation*}
\end{theorem}

\begin{remark}
   Since $X({\bf0})$ is a Bernoulli random variable, then $\| X({\bf0})\|_{2+\delta}=\left(  \int X({\bf0})^{2+\delta}d\mathbb{P}  \right)^{1/(2+\delta)}<\infty$ for all $\delta>0$.  Besides, Ritchie~\cite[Theorem 51]{MR2205908} established the super-exponential decay of correlations in the thermodynamic limit $X$, that is,
    \begin{equation}\label{correlations}
        \lim_{\|i\|_{\max}\to\infty}\mathrm{Cov}( X({\bf0}),X(i)  )\cdot e^{\alpha \| i\|_{\max}}=0,\forall \alpha>0.
    \end{equation}
    This directly implies  statements \eqref{mainresult1} and \eqref{mainresult12} of the above theorems (see the beginning of the proof of Lemma \ref{Var} below).
\end{remark}

\subsection{Technical lemmas}\label{sec:techlemmas}

In order to prove Theorem~\ref{main}, we will verify the convergence criteria given by the previous theorems. This is the object of the following lemmas.

\begin{lemma}\label{Var}
    Let $N_n$ be the number of occupied sites in $\Lambda_n$ for the thermodynamic jamming limit $X$. Then
    \begin{equation*}
        \lim_{n\to\infty}\frac{\mathrm{Var} \,(N_n)}{|\Lambda_n|}=\sum_{i\in\mathbb{Z}^2}\mathrm{Cov}\big(X({\bf0}),X(i)  \big).
    \end{equation*}
\end{lemma}
\begin{proof}
    The proof of this lemma is analogous to the proof of Proposition 7.2 in \cite{MR2433297}. We include it here for the sake of completeness.
   
    It follows from (\ref{correlations}) that for any $\alpha>0$ there exists $C_\alpha > 0$ such that for any $i,j\in\mathbb{Z}^2$
    \begin{equation*}
        |\mathrm{Cov}(X(i),X(j))|\leq C_\alpha e^{-\alpha ||i-j||_{\max}}.
    \end{equation*}
Thus
    \begin{eqnarray*}
        \sum_{i\in \mathbb{Z}^2} |\mathrm{Cov}(X({\bf0}),X(i))| &\leq & C_{\alpha} \sum_{i\in\mathbb{Z}^2}e^{-\alpha ||i||_{\max}} \\
        &\leq & C_{\alpha} \sum_{i\in \mathbb{Z}^2}\sum_{r=0}^\infty e^{-\alpha ||i||_{\max}} \1_{\{||i||_{\max}=r \}} \\
        &= & C_\alpha \sum_{r=0}^\infty e^{-\alpha r}\sum_{i\in \mathbb{Z}^2} \1_{\{||i||_{\max}=r \}}  \\
        &= & C_\alpha \sum_{r=0}^\infty e^{-\alpha r}|\partial\Lambda_{r}|.\\
%        &\leq & C_{\alpha,\gamma}\sum_{r=0}^{\infty} e^{-(\alpha-\gamma)r} <\infty.
    \end{eqnarray*}
   
Since $|\partial\Lambda_{r}|$ increases only polynomially, we conclude that
    \begin{equation}\label{eq:sumcovfinite}
            \sum_{i\in \mathbb{Z}^2} |\mathrm{Cov}(X({\bf0}),X(i))|<\infty.
    \end{equation}
   
Consider a real sequence $u_n\rightarrow\infty$ such that
    \begin{equation*} 
        \lim_{n\to \infty}\frac{u_n |\partial \Lambda_n|}{|\Lambda_n|}=0.
    \end{equation*}

    Let $B_n:=\{ i\in \Lambda_n: d(\{i\},\partial \Lambda_n)<u_n  \}$ for $n\in\mathbb{N}$ and note that  $|B_n|\leq |\partial \Lambda_n| u_n$ thus
    \begin{equation} \label{limit}
        \lim_{n\to \infty}\frac{|B_n|}{|\Lambda_n|}=0.
    \end{equation}
%:
    Set $F_n=\Lambda_n\setminus B_n$ and recall that $ \Lambda_{u_n}(i)$ denotes the box $\{j\in\mathbb{Z}^2: {\|i-j\|_{\max}}\leq u_n \}$. Thanks to the absolute summability \eqref{eq:sumcovfinite}, we can reorder the sum as follows
        \begin{eqnarray*}
        \frac{\mathrm{Var}(N_n)}{|\Lambda_n|}&=& \frac{1}{|\Lambda_n|}\sum_{i\in \Lambda_n}\sum_{j\in \Lambda_n}\mathrm{Cov}(X(i),X(j))\\
        &=& T_{1,n}+T_{2,n}+T_{3,n}
    \end{eqnarray*}
    where
    \begin{eqnarray*}    
        T_{1,n}&=&\frac{1}{|\Lambda_n|}\sum_{i\in F_n}\sum_{j\in \Lambda_n\setminus \Lambda_{u_n}(i)}\mathrm{Cov}(X(i),X(j)), \\
        T_{2,n}&=&\frac{1}{|\Lambda_n|}\sum_{i\in F_n}\sum_{j\in \Lambda_n\cap \Lambda_{u_n}(i) }\mathrm{Cov}(X(i),X(j)), \\
        T_{3,n}&=&\frac{1}{|\Lambda_n|}\sum_{i\in  B_n}\sum_{j\in \Lambda_n }\mathrm{Cov}(X(i),X(j)).
    \end{eqnarray*}
    Now observe that by \eqref{eq:sumcovfinite}, we get
    \begin{equation*}
    |T_{1,n}|\le \frac{| F_n|}{|\Lambda_n|}\sum_{j\in \Lambda_{u_n}^c }|\mathrm{Cov}(X({\bf0}),X(j))|=\frac{| F_n|}{|\Lambda_n|}\sum_{j: \|j\|_{\max}\ge u_n}|\mathrm{Cov}(X({\bf0}),X(j))|\longrightarrow0.
    \end{equation*}
On the other hand,
    \begin{equation*}
        |T_{3,n}|\leq \frac{| B_n|}{|\Lambda_n|}\sup_{i\in\mathbb{Z}^2}\sum_{j\in\mathbb{Z}^2}|\mathrm{Cov}(X(i),X(j))|\longrightarrow0
    \end{equation*}
by \eqref{limit} and \eqref{eq:sumcovfinite}. Finally, by \eqref{limit} we have that
\begin{eqnarray*}
\lim_{n \rightarrow +\infty} T_{2,n} &=& \lim_{n \rightarrow +\infty} \frac{|F_n|}{|\Lambda_n|} \sum_{i\in \Lambda_{u_n}}\mathrm{Cov}(X({\bf0}),X(i)) = \sum_{i\in \mathbb{Z}^2} \mathrm{Cov}(X({\bf0}),X(i))
\end{eqnarray*}
and the proof is complete.
\end{proof}

The next lemma provides the asymptotic behavior of the mixing coefficient $\alpha_{k,l}(n)$.

\begin{lemma}\label{lemma1} Let $k,l \in \mathbb{N}$ be fixed, then the mixing coefficient $\alpha_{k,l}(n)$ satisfies
\begin{equation*}
\alpha_{k,l}(n)\le (k+l)\frac{3^{\lceil n/4\rceil}}{\lfloor n/4\rfloor ! }.
\end{equation*}
\end{lemma}

\begin{proof}
The main idea of the proof of this lemma consists in analyzing three versions of the thermodynamic jamming limit given by a suitable coupling. Let $U=\{U(i) \}_{i\in\mathbb{Z}^2}$ and $V=\{V(i) \}_{i\in\mathbb{Z}^2}$ be two families of mutually independent i.i.d. uniform random variables on $[ 0,1 ]$. Let $R_1=\{ (i_1,i_2)\in \mathbb{Z}^2:i_1\leq 0 \}$ and $R_2=\{ (i_1,i_2)\in \mathbb{Z}^2:i_1> 0 \}$. Consider the following random fields
\begin{eqnarray*}
Y(i) &=& U(i)\1_{\{i\in R_1\}}+V(i)\1_{\{i\in R_2\}}, \\
Z(i) &=& V(i)\1_{\{i\in R_1\}}+U(i)\1_{\{i\in R_2\}}.
\end{eqnarray*}
Notice that $Y=\{Y(i)\}_{i\in\mathbb{Z}^2}$ and $Z=\{Z(i)\}_{i\in\mathbb{Z}^2}$ are two families of uniform i.i.d. random variables on $(0,1)$ which are mutually independent.
Recall the notation $X(U),X(Y)$ and $X(Z)$ to denote the versions of the thermodynamic jamming limit obtained using the random fields $U,Y$ and $Z$ respectively. Notice that $X(Y)$ and $X(Z)$ are independent. Let
$G,H$ be two finite subsets of $\mathbb{Z}^2$ with $|G|\leq k$ and $|H|\leq l$  such that
\begin{equation*}
\max\{ i_1: (i_1,i_2)\in G\}\leq -2n \hspace{.2cm} \text{and} \hspace{.2cm} \min\{ i_1: (i_1,i_2)\in H \}\geq 2n.
\end{equation*}
Thus $d(G,H)\geq 4n$. Let $A\in\mathcal{F}_{G}$ and $B\in\mathcal{F}_{H}$. Consider the  event
\begin{equation*}
E=\bigcap_{i\in G\cup H}\{ \mathcal{A}(\{i\})(U)\subset \Lambda_n(i) \}.
\end{equation*}
Note that
\begin{equation*}
\1_{A\cap B}^U=\1_A^Y{\1}_B^Z \hspace{.2cm}\text{on $E$},
\end{equation*}
where $\1_A^{U}$ is a shorthand notation for the indicator function of the event $X(U)\in A$.  Then,
\begin{eqnarray*}
|\mathbb{P}(A \cap B)-\mathbb{P}(A) \mathbb{P}(B) | & = & |\mathbb{E}(\1_{A\cap B}^{U})-\mathbb{E}(\1_A^{Y}\1_B^{Z})| \\
& \leq & \mathbb{E}(|\1_{A\cap B}^{U}-\1_A^{Y}\1_B^{Z}|) \\
& \leq & \mathbb{P}(E^c).
\end{eqnarray*}
 It follows that
\begin{align*}
\mathbb{P}(E^c)&=\mathbb P\left(\bigcup_{i\in G\cup H}\{ \mathcal{A}(\{i\})(U)\not\subset \Lambda_{n}(i) \}\right)\\
&\le\sum_{i\in G\cup H}\mathbb P( \mathcal{A}(\{i\})(U)\not\subset \Lambda_{n}(i) )\\
&\leq (|G|+|H|)\mathbb P( \mathcal{A}(\{0\})(U)\not\subset \Lambda_n ).
\end{align*}
Now, observe that for $n\ge3$
\begin{align}\label{eq:armourn}
\mathbb P(\mathcal A(0)\not\subset \Lambda_n)&=\mathbb P(\exists i\in \Lambda_n^c:U(0\downarrow i))\le \frac{4\cdot 3^n}{(n+1)!}\le \frac{3^n}{n!}.
\end{align}
The $4\cdot 3^n$ is an upper bound on the number of self-avoiding path of size $n+1$ starting from 0 (four possibilities for the first step and at most 3 possibilities for the subsequents steps), and $\frac{1}{(n+1)!}$ is the probability of any such path  of size $n+1$ with decreasing uniforms. We only use the restriction $n\ge3$ to simplify to the last inequality.

We conclude that
\begin{equation*}
\alpha_{k,l}(n)\le (k+l)\left(\frac{3^{\lceil n/4\rceil}}{\lfloor n/4\rfloor ! }\right),
\end{equation*}
which is the desired result.
\end{proof}

Next lemma provides the asymptotic behavior of the mixing coefficient $\alpha_{1,\infty}(n)$.
\begin{lemma}\label{lemma2}
The mixing coefficient $\alpha_{1,\infty}(n)$ satisfies
\begin{equation*}
\alpha_{1,\infty}(n)\le  \frac{(16n+1)3^n}{n!}.
\end{equation*}
\end{lemma}

\begin{proof}

The proof is very similar to the proof of the preceding lemma. Let $\{U(i)\}_{i\in\mathbb{Z}^2}$ and $\{V(i)\}_{i\in\mathbb{Z}}$ be two independent families of i.i.d. random variables with a uniform distribution on $(0,1)$. Consider the following random fields
\begin{eqnarray}
Y(i) &=& U(i)\1_{\{i\in \Lambda_{2n}\}}+V(i)\1_{\{i\notin \Lambda_{2n}\}},
\nonumber \\
Z(i) &=& V(i)\1_{\{i\in \Lambda_{2n}\}}+U(i)\1_{\{i\notin \Lambda_{2n}\}}. \nonumber
\end{eqnarray}

Note that $\{Y(i)\}_{i\in\mathbb{Z}^2}$ and $\{Z(i)\}_{i\in\mathbb{Z}^2}$ are two families of independent uniform random variables which are also independent between them.

$X(U),X(Y)$ and $X(Z)$ to denote the versions of the thermodynamic jamming limit obtained using the random fields $U,Y$ and $Z$ respectively. Notice that $X(Y)$ and $X(Z)$ are independent.

Let $A\in \mathcal{F}_{\{0\}}$ and $B\in\mathcal{F}_{\Lambda^c_{2n}}$ and observe that $d(\{0\},\Lambda^c_{2n})\geq 2n$. Next let us define the event
\begin{equation*}
E=\{  \mathcal{A}(\{0\})\subset \Lambda_{n}  \}\cap \{  \mathcal{A}(\partial \Lambda_{2n})\cap \Lambda_n=\emptyset  \}
\end{equation*}
and observe that
\begin{equation*}
\1_{A\cap B}^{U}=\1_A^Y \1_B^Z \hspace{.2cm}\text{on $E$}.
\end{equation*}
Thus
\begin{eqnarray*}
|\mathbb{P}(A \cap B)-\mathbb{P}(A) \mathbb{P}(B) |=|\mathbb{E}(\1^U_{A\cap B})-\mathbb{E}(\1_A^{Y}\1_B^{Z})|\leq  \mathbb{E}(|\1^U_{A\cap B}-\1_A^{Y}\1_B^{Z}|)\leq \mathbb{P}(E^c).
\end{eqnarray*}
Since $\mathbb{P}(E^c)\leq \frac{3^n}{n!}+|\partial \Lambda_{2n}|\frac{3^n}{n!}=\frac{(16n+1)3^n}{n!}$ we conclude that
\begin{equation*}
\alpha_{1,\infty}(n)\le \frac{(16n+1)3^n}{n!}.
\end{equation*}
\end{proof}

\subsection{Asymptotic variance}
Next we prove that the asymptotic variance for the parking process is non-trivial. The proof presented follows the ideas described in Penrose~\cite{MR1887532}.

\begin{lemma}\label{lemma3}
$\lim_{n\to\infty}\frac{\mathrm{Var} \,(N_n)}{|\Lambda_n|}=\sigma^2>0$.
\end{lemma}

\begin{proof}
Let us consider the following partition of boxes centered at the origin. For $n\geq1$, the box $\Lambda_{7n+3}$ can be partitioned into $(2n+1)^2$ boxes of dimensions $7\times7$. Specifically, $\Lambda_{7n+3}$ is the union of the disjoint boxes
\begin{equation*}
\Lambda_{\kappa}:= \Lambda_3(7\kappa), \, \text{ with } \, \kappa\in\Lambda_n.
\end{equation*}

Denote by $A$ and $B$ the following subsets of the box $\Lambda_3$
\begin{eqnarray*}
A &=& \{(1,2), (1,-2),(2,1), (2,-1), (-1,2),(-1, -2), (-2,1), (-2,-1)\},
\\
B &=& \{(-1,0),(0,1),(0,0),(0,-1),(1,0) \}.
\end{eqnarray*}

\begin{figure}[h!]
\centering
\includegraphics[width=6cm]{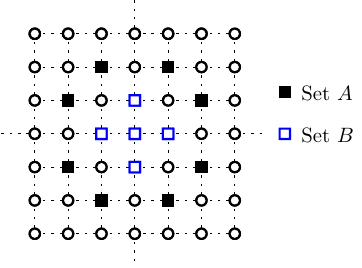}
\label{Lambda3}
\caption{Representation of $\Lambda_3$ and its subsets $A$ and $B$.}
\end{figure}

Also, for each $\Lambda_{\kappa}$, we define the random variables
\begin{equation*}
U_{\kappa}:=\max\{U(i): i\in A + 7\kappa\}
\end{equation*}
and
\begin{equation*}
V_{\kappa}:=\min\{U(i): i\in (\Lambda_3 \setminus A)+ 7\kappa \}.
\end{equation*}
We say that the box $\Lambda_{\kappa}$ is a \textit{good box} if the event $\{U_{\kappa}<V_{\kappa}\}$ occurs.
 Let $I_{\kappa}:={\1}_{\{\Lambda_{\kappa}\text{ is good}\}}$ for ${\kappa}\in\Lambda_n$  which are independent Bernoulli random variables with a probability of success $\beta:=\mathbb{P}(U_{\kappa}<V_{\kappa})>0$. Consider now the $\sigma$-algebra $\mathcal{F}_{7n+3}$ generated by the value of $\sum_{\kappa\in\Lambda_n}I_{\kappa}$  and the uniform random variables of all the sites in $\Lambda_{7n+3}$ except those from $\{B+7\kappa: \kappa\in\Lambda_n \text{ and } I_{\kappa}=1\}.$  

Observe that
\begin{equation*}
\mathrm{Var}(N_{7n+3})=\mathrm{Var}[\mathbb{E}\big(N_{7n+3}\big|\mathcal{F}_{7n+3}\big)]+\mathbb{E}[\mathrm{Var}\big(N_{7n+3}\big|\mathcal{F}_{7n+3}\big)]\ge\mathbb{E}[\mathrm{Var}\big(N_{7n+3}\big|\mathcal{F}_{7n+3}\big)].
\end{equation*}
Writing
\[
N_{7n+3}=\left(N_{7n+3}-\sum_{\kappa\in\Lambda_n}I_{\kappa}\sum_{i\in B+7\kappa}X(i)\right)+\sum_{\kappa\in\Lambda_n}I_{\kappa}\sum_{i\in B+7\kappa}X(i),
\]
we notice that the random variable between parentheses is $\mathcal{F}_{7n+3}$-mensurable and therefore
\begin{align*}
\mathrm{Var}\big(N_{7n+3}\big|\mathcal{F}_{7n+3}\big)&=\mathrm{Var}\left( \sum_{\kappa\in\Lambda_n}I_{\kappa}\sum_{i\in B+7\kappa}X(i)\, \big| \,\mathcal{F}_{7n+3}\right).
\end{align*}
Now, we use the fact that, conditionally on $\mathcal{F}_{7n+3}$, the random variables $\sum_{i\in B+7\kappa}X(i)$ for those $\kappa$'s such that $I_{\kappa}=1$, are independent  and identically distributed. Thus,
\begin{align*}
\mathrm{Var}\big(N_{7n+3}\big|\mathcal{F}_{7n+3}\big)&=\sum_{\kappa\in\Lambda_n}I_{\kappa}\mathrm{Var}\left( \sum_{i\in B+7\kappa}X(i)\big|\mathcal{F}_{7n+3}\right).
\end{align*}
If $I_{\kappa}=1$, then $\sum_{i\in B+7\kappa}X(i)=1$ with probability $1/5$ and it is $4$ with probability $4/5$. Then, for those $\kappa$'s such that $I_{\kappa}=1$,
$\mathrm{Var}\left( \sum_{i\in B+7\kappa}X(i)\big|\mathcal{F}_{7n+3}\right)=\alpha$ for some constant $\alpha>0$.  

Putting everything together, we obtain that
\[
\mathrm{Var}(N_{7n+3})\ge \alpha \mathbb{E}\left(\sum_{\kappa\in\Lambda_n}I_{\kappa}\right)=(2n+1)^2\alpha\beta.
\]

Thus,  
\begin{equation*}
\frac{\mathrm{Var}(N_{7n+3})}{|\Lambda_{7n+3}|}\geq\frac{\alpha\beta(2n+1)^2}{(2(7n+3)+1)^2}.
\end{equation*}
This, in turn, implies that
\begin{equation*}
\liminf_{n\to\infty}\frac{\mathrm{Var}(N_{7n+3})}{|\Lambda_{7n+3}|}>0
\end{equation*}
and the proof is complete.
\end{proof}

We are now ready to prove Theorem \ref{main}.

\subsection{Proof of Theorem \ref{main}}\label{proof_mainthm}

\begin{proof}
For the central limit theorem statement we check that we are in force of all the conditions of Theorem \ref{ThmMain}:
\begin{itemize}
\item Item (a) is satisfied, by Lemma \ref{lemma1}.
\item Item (b) is satisfied also by Lemma \ref{lemma1}, taking $\delta=1$ for instance (recalling that $X(i)\in\{0,1\}$ in our case).
\item Item (c) is satisfied by  Lemma \ref{lemma2}.
\end{itemize}
From Lemmas \ref{Var} and \ref{lemma3} we have that
\begin{equation}\label{eqm}
\lim_{n\to\infty}\frac{\mathrm{Var} \,(N_n)}{|\Lambda_n|}=\sum_{i\in\mathbb{Z}^2}\mathrm{Cov}(X(0),X(i))=\sigma^2>0.
\end{equation}
So Theorem \ref{ThmMain} applied and the CLT part of the theorem is proved.

For the law of iterated logarithm, we check that we are in force of all the conditions of Theorem \ref{nahap}:
 \begin{itemize}
\item Item (a) is satisfied with $\delta=1$ for instance, since $X(i)\in\{0,1\}$.
\item Item (b) is satisfied  by Lemma \ref{lemma1}, taking $\tau_1=\tau_2=1$.
\item Item (c) is satisfied by  Lemma \ref{lemma1} (take $\delta=1$, $\delta'=1/2$, and $\beta=5$).
\end{itemize}
Finally, \eqref{eqm} above also proves the variance condition here. So Theorem \ref{nahap} applied and the LIL part of the theorem is proved, concluding the proof of Theorem \ref{main}.
\end{proof}

\subsection{Proof of Proposition \ref{prop:concentration}}\label{sec:proof_conc}

Let us first define the mixing coefficients.
\begin{equation*}
\phi_{\infty,1}(k):=\sup\{|\mathbb P(X({\bf0})=1|A)-\mathbb P(X({\bf0})=1)|:A\in\mathcal F_{\Lambda_k^c}, \mathbb P(A)>0\}.
\end{equation*}
We will use \cite[Corollary 4(b)(i)]{dedecker2001exponential}, which states that, in our framework, if $B=1+\sum_{k\in\mathbb Z^d\setminus\{0\}}\phi_{\infty,1}(||k||)<\infty$, then
\[
\mathbb P\left(\left|\frac{N_n}{|\Lambda_n|}-\rho\right|>\epsilon\right)\le e^{\frac{1}{e}-\frac{|\Lambda_n|\epsilon^2}{4Be}},\epsilon>0.
\]

\begin{proof}
According to what we mentioned above,  we only have to compute $B$ (and in particular show that it is finite.  First observe that for any $\Lambda\subset\mathbb Z^2$ and any $x$ we have
\[
\mathbb P(X({\bf0})=1|X_{\Lambda^c}=x_{\Lambda^c})=\mathbb P(X^{(x)}_\Lambda({\bf0})=1).
\]
On the other hand, $X^{(x)}_\Lambda$ can be constructed on the same probability space as $X$, that is, using the same field $U$. Thus
\begin{align*}
|\mathbb P(X({\bf0})=1|X_{\Lambda_k^c}=x_{\Lambda_k^c})-\mathbb P(X({\bf0})=1)|&\leq \mathbb P(X({\bf0})(U)\ne X^{(x)}_{\Lambda_k}({\bf0})(U))\\
&\le \mathbb P(\mathcal A({\bf 0})(U)\cap \partial\Lambda_k=\emptyset)\\&\le \frac{2d(2d-1)^{k-1}}{k!}.
%\le \int_\mathbb Ar(x,y,k)d\mathbb P(y)
\end{align*}
(The last inequality follows using the same argument that yields to \eqref{eq:armourn}, but for $d\ge1$).

We therefore get for any $A\in\mathcal F_{\Lambda_k^c}$ with $\mathbb P(A)>0$
\begin{align*}
|\mathbb P(X({\bf0})=1|A)-\mathbb P(X({\bf0})=1)|&=\frac{1}{\mathbb P(A)}|\mathbb P(\{X({\bf0})=1\}\cap A)-\mathbb P(X({\bf0})=1)\mathbb P(A)|\\
&\le\frac{1}{\mathbb P(A)}\int_{A}|\mathbb P(X({\bf0})=1| X_{\Lambda_k^c}=x_{\Lambda_k^c})-\mathbb P(X({\bf0})=1)|d\mathbb P(x)\\
%&\le\frac{1}{\mathbb P(A)}\int_{A}\frac{3^{2k}}{k!}d\mathbb P(x)\\
&\le \frac{2d(2d-1)^k}{(2d-1)k!},
\end{align*}
which is  an upper bound for $\phi_{\infty,1}(k)$.

We conclude with
\begin{align*}
B&=1+\sum_{k\in\mathbb Z^d\setminus\{0\}}\phi_{\infty,1}(\|k\|)\\
&\le 1+\sum_{k\ge1}\sum_{v\in\partial\Lambda_k}\frac{2d(2d-1)^k}{(2d-1)k!}\\
&= 1+\frac{2d}{2d-1}\sum_{k\ge1}\frac{(2d-1)^{k}[(2k+1)^d-(2k-1)^d]}{k!}.
\end{align*}
\end{proof}

\subsection{Proof of Proposition \ref{prop:mean-dev}}\label{sec:proof_expect}
\begin{proof}
We have
\begin{align*}
\left|\,\mathbb E \bar N_n-|\Lambda_n|\rho\,\right|&=\left|\,\mathbb E (\bar N_n-N_n)\,\right|\le\mathbb E\left|\,\bar N_n-N_n\,\right|\\
&\le\mathbb E\sum_{i\in\Lambda_n}\1_{\{X(i)(U)\ne X_{\Lambda_n}(i)(U)\}}\\
&\le \sum_{i\in\Lambda_n}\mathbb P(X(i)(U)\ne X_{\Lambda_n}(i)(U))\\
&\le \sum_{r=0}^n\sum_{i\in\partial\Lambda_r}\mathbb P(X(i)(U)\ne X_{\Lambda_n}(i)(U))\\
&\le \sum_{r=0}^n\sum_{i\in\partial\Lambda_r}\mathbb P(\mathcal A(\{0\})\not\subset \Lambda_{n-r})\\
&\le \sum_{r=0}^n\sum_{i\in\partial\Lambda_r}\frac{2d(2d-1)^{n-r}}{(n-r+1)!}\\
&= \sum_{r=0}^n|\partial\Lambda_{n-r}|\frac{2d(2d-1)^{r}}{(r+1)!}\\
&= \frac{2d(2d-1)^{n}}{(n+1)!}+\sum_{r=0}^{n-1}\frac{2d(2d-1)^{r}[(2(n-r)+1)^d-(2(n-r)-1)^d]}{(r+1)!}\\
&\le \frac{2d(2d-1)^{n}}{(n+1)!}+(2d)^2\sum_{r=0}^{n-1}\frac{(2d-1)^{r}(2(n-r)+1)^{d-1}}{(r+1)!}.
\end{align*}
The last inequality follows from an application of the mean value theorem to $f(x)=x^d$. Notice that this inequality yields $4(e-1)$ for $d=1$, twice the stated value of the proposition. The calculation for $d=1$ can be done slightly more easily: 
\begin{align*}
\left|\,\mathbb E \bar N_n-|\Lambda_n|\rho\,\right|&\le \sum_{i\in\Lambda_n}\mathbb P(X(i)(U)\ne X_{\Lambda_n}(i)(U))\le 2\sum_{r=0}^{n}\frac{1}{(r+1)!}\le 2(e-1).
\end{align*}
\end{proof}

\subsection{Proof of Theorem \ref{density1}}\label{sec:proof_prop}
\begin{proof}
Let $\mathcal{A}(\{0\})$ be the armour of $\{0\}$. Since $\mathcal{A}(\{0\})$ is almost surely finite (see \cite[Lemma 31]{MR2205908}), then the armour $\mathcal{A}(\{0\})$ may be anyone of the following (random) sets:
\begin{itemize}
    \item[$(a)$] $\mathcal{A}(\{0\})=\{0\}$;
    \item[$(b)$] $\mathcal{A}(\{0\})=\{0, \ldots, n\}$, for some $n\in\mathbb{N}$;
    \item[$(c)$] $\mathcal{A}(\{0\})=\{-m, \ldots, 0\}$, for some $m\in\mathbb{N}$;
    \item[$(d)$] $\mathcal{A}(\{0\})=\{-m, \ldots, n\}$, for some $m, n\in\mathbb{N}$.\\
\end{itemize}

Let $U=\{U_i\}_{i\in\mathbb{Z}}$ be the i.i.d. random variables with a uniform distribution on $(0,1)$ used to build the armours.

If $\mathcal{A}(\{0\})=\{0\}$, then

\begin{equation}
        \mathbb{P}(\mathcal{A}(\{0\})=\{0\})=\mathbb{P}(U_0<\min\{U_{-1},U_1\})=\frac{1}{3}.\label{armour1}
    \end{equation}

If, for some $n \in \mathbb{N}$, the random set $\mathcal{A}(\{0\})$ equals $\{0, \ldots , n \}$, then

\begin{eqnarray*}
        \mathbb{P}(\mathcal{A}(\{0\})=\{0, \ldots, n\})&=&\mathbb{P}(U_{-1}>\cdots>U_n, \, U_n<U_{n+1})\\
        &=&\mathbb{P}(\{U_{-1}>\cdots>U_n\} \setminus \{U_{-1}>\cdots>U_{n+1}\})\\
        &=&\mathbb{P}(U_{-1}>\cdots>U_n)-\mathbb{P}(U_{-1}>\cdots>U_{n+1})\\
        &=&\frac{1}{(n+2)!}-\frac{1}{(n+3)!}\\
        &=&\frac{n+2}{(n+3)!}.
    \end{eqnarray*}
    Thus,  
    \[
    \sum_{n\geq1}\mathbb{P}(\mathcal{A}(\{0\})=\{0, \ldots, n\})=\sum_{n\geq1}\left[\frac{1}{(n+2)!}-\frac{1}{(n+3)!}\right]=\frac{1}{6} \ .
    \]

Note that, analogously to the previous case,
\[
\mathbb{P}(\mathcal{A}(\{0\})=\{-m, \ldots, 0\})=\frac{m+2}{(m+3)!}
\]
and
    \[
    \sum_{m\geq1}\mathbb{P}(\mathcal{A}(\{0\})=\{-m, \ldots, 0\})=\frac{1}{6}.
    \]

Finally, if for some $m,n \in \mathbb{N} \ \mathcal{A}(\{0\})=\{-m, \ldots, n\}$, then

\begin{eqnarray}
        &&\mathbb{P}(\mathcal{A}(\{0\})=\{-m, \ldots, n\})\nonumber\\
        &=&\mathbb{P}(U_{-m-1}>U_{-m}, \, U_{-m}<\cdots<U_0, \, U_0>\cdots >U_n, \, U_n<U_{n+1})\nonumber\\
        &=&\int_0^1 \mathbb{P}(U_{-m-1}>U_{-m}, \, U_{-m}<\cdots<U_{-1}<u, \, u>U_1>\cdots >U_n, \, U_n<U_{n+1}) \ du  \nonumber\\
        &=&\int_0^1 \mathbb{P}(U_{-m-1}>U_{-m}, \, U_{-m}<\cdots<U_{-1}<u) \mathbb{P}(u>U_1>\cdots >U_n, \, U_n<U_{n+1}) \ du  \nonumber \\
        &=&\int_0^1 \left(\mathbb{P}(U_{-m}<\cdots<U_{-1}<u)-\mathbb{P}(U_{-m-1}<\cdots<U_{-1}<u)\right)\cdot \nonumber\\
        &&\quad\quad\quad\quad\quad\quad\quad\left(\mathbb{P}(U_{n}<\cdots<U_{1}<u)-\mathbb{P}(U_{n+1}<\cdots<U_{1}<u)\right) du \nonumber \\
    &=&\displaystyle\int_0^1\left(\frac{u^m}{m!}-\frac{u^{m+1}}{(m+1)!}\right)\left(\frac{u^n}{n!}-\frac{u^{n+1}}{(n+1)!}\right) du \label{repintegral} \\   \nonumber\\
        &=&(b_{m,n}-b_{m,n+1})-(b_{m+1,n}-b_{m+1,n+1}),\nonumber
    \end{eqnarray}
where $b_{m,n}:=\frac{1}{(m+n+1)m!n!}$.  Thus,
    \begin{eqnarray*}
        \sum_{n\geq1}\sum_{m\geq1}\mathbb{P}(\mathcal{A}(\{0\})=\{-m, \ldots, n\})&=&\sum_{n\geq1}\sum_{m\geq1}[(b_{m,n}-b_{m,n+1})-(b_{m+1,n}-b_{m+1,n+1})]\\
        &=&\sum_{n\geq1}[b_{1,n}-b_{1,n+1}]\\
        &=&b_{1,1}=\frac{1}{3}.
    \end{eqnarray*}
According to the construction of the thermodynamic limit $X$, we have the following scenarios:

\begin{enumerate} [(i)]
\item If $\mathcal{A}(\{0\})=\{0\},$ then $X(0)=1.$\\
\item If $\mathcal{A}(\{0\})=\{0, \ldots, n\}$ for some $n\in\mathbb{N}$, then
\[
X(0)=
\begin{cases}
    1, & \text{if $n$ is even} \\
    0, & \text{if $n$ is odd.}
\end{cases}
\]
\item If $\mathcal{A}(\{0\})=\{-m, \ldots, 0\}$ for some $m\in\mathbb{N}$, then
\[
X(0)=
\begin{cases}
    1, & \text{if $m$ is even} \\
    0, & \text{if $m$ is odd.}
\end{cases}
\]
\item If $\mathcal{A}(\{0\})=\{-m, \ldots, n\}$ for some $m,n\in\mathbb{N}$, then
\[
X(0)=
\begin{cases}
    1, & \text{if $m$ and $n$ are even} \\
    0, & \text{if $m$ or $n$ are odd.}
\end{cases}
\]
\end{enumerate}

It follows from the discussion above that

\begin{eqnarray}
\mathbb{P}(X(0)=1)
%&=&\mathbb{P}[\mathcal{A}(\{0\})=\{0\}]
%+\mathbb{P}\left[\bigcup_{n\geq1}\{\mathcal{A}(\{0\})=\{0, \ldots,2n\}\}\right]
%+\mathbb{P}\left[\bigcup_{m\geq1}\{\mathcal{A}(\{0\})=\{-2m, \ldots,0\}\}\right] \nonumber\\
%&&+\mathbb{P}\left[\bigcup_{n\geq1}\bigcup_{m\geq1}\{\mathcal{A}(\{0\})=\{-2m, \ldots,2n\}\}\right] \nonumber\\\\
&=&\mathbb{P}(\mathcal{A}(\{0\})=\{0\})
+\sum_{n\geq1}\mathbb{P}\left(\mathcal{A}(\{0\})=\{0, \ldots,2n\}\right)
+\sum_{m\geq1}\mathbb{P}\left(\mathcal{A}(\{0\})=\{-2m, \ldots,0\}\right) \nonumber\\
&&+\sum_{n\geq1}\sum_{m\geq1}\mathbb{P}\left(\mathcal{A}(\{0\})=\{-2m, \ldots,2n\}\right). \label{P0}
\end{eqnarray}
Note that
\begin{eqnarray}
\sum_{m\geq1}\mathbb{P}\left(\mathcal{A}(\{0\})=\{-2m, \ldots,0\}\right)&=&
\sum_{n\geq1}\mathbb{P}\left(\mathcal{A}(\{0\})=\{0, \ldots,2n\}\right)\nonumber\\
&=&\sum_{n\geq1}\frac{2n+2}{(2n+3)!}\nonumber\\
&=&\sum_{n\geq1}\left[\frac{1}{(2n+2)!}-\frac{1}{(2n+3)!}\right]\nonumber\\
&=&\sum_{n\geq4}\frac{(-1)^n}{n!}\nonumber\\
&=&\frac{1}{e}-\frac{1}{3}.\label{armour2}
\end{eqnarray}

On the other hand, it follows from (\ref{repintegral}) that

\

$\displaystyle\sum_{n\geq1}\sum_{m\geq1}\mathbb{P}\left(\mathcal{A}(\{0\})=\{-2m, \ldots,2n\}\right)$\\
\begin{eqnarray}
&=&\displaystyle\sum_{n\geq1}\sum_{m\geq1}\int_0^1\left(\frac{u^{2m}}{(2m)!}-\frac{u^{2m+1}}{(2m+1)!}\right)\left(\frac{u^{2n}}{(2n)!}-\frac{u^{2n+1}}{(2n+1)!}\right) du \nonumber\\ \nonumber\\
&=&\displaystyle\sum_{n\geq1}\sum_{m\geq1}\int_0^1\left(\frac{u^{2m+2n}}{(2m)!(2n)!}-\frac{u^{2m+2n+1}}{(2m)!(2n+1)!}-\frac{u^{2m+2n+1}}{(2m+1)!(2n)!}+\frac{u^{2m+2n+2}}{(2m+1)!(2n+1)!}\right) du\nonumber\\ \nonumber\\
&=&\displaystyle\int_0^1\left[ \left(\sum_{m\geq1}\frac{u^{2m}}{(2m)!}\sum_{n\geq1}\frac{u^{2n}}{(2n)!}\right)
-\left(\sum_{m\geq1}\frac{u^{2m}}{(2m)!}\sum_{n\geq1}\frac{u^{2n+1}}{(2n+1)!}\right)\right.\nonumber\\ \nonumber\\
&&\left.-\displaystyle\left(\sum_{m\geq1}\frac{u^{2m+1}}{(2m+1)!}\sum_{n\geq1}\frac{u^{2n}}{(2n)!}\right)
+\left(\sum_{m\geq1}\frac{u^{2m+1}}{(2m+1)!}\sum_{n\geq1}\frac{u^{2n+1}}{(2n+1)!}\right)\right]du  \nonumber\\ \nonumber\\
&=&\displaystyle\int_0^1 \left[(\cosh(u)-1)^2 -2(\cosh(u)-1)(\sinh(u)-u)+(\sinh(u)-u)^2\right]du \nonumber\\ \nonumber\\
%&=&\displaystyle\int_0^1 \left[(\cosh(u)-1) -(\sinh(u)-u)\right]^2du \\ \\
%&=&\displaystyle\int_0^1 \left[e^{-u}+u-1)\right]^2du \\ \\
&=&\displaystyle\frac{5}{6}-\frac{2}{e}-\frac{1}{2e^2}. \label{armour3}\\ \nonumber
\end{eqnarray}

Collecting (\ref{armour1}), (\ref{armour2}) and (\ref{armour3}) in (\ref{P0}), we obtain
\begin{eqnarray}
\mathbb{E}[X(0)] &=&\mathbb{P}(X(0)=1) \nonumber\\
&=& \frac{1}{3} + 2 \left(\frac{1}{e} - \frac{1}{3} \right) + \left( \frac{5}{6} - \frac{2}{e} - \frac{1}{2e^2}\right) \nonumber \\
&=& \frac{1-e^{-2}}{2} \nonumber
\end{eqnarray}
which is the desired result.  
\end{proof}

\begin{proof}[Proof of Proposition \ref{prop:d=1}] Let  $\sigma^2>0$ be as given in Lemma~\ref{lemma3}, and write 
\[\frac{\bar N_n-\mathbb E(\bar N_n)}{\sqrt{2|\Lambda_n|\sigma^2\log\log|\Lambda_n|}}=
\frac{ N_n-|\Lambda_n|\rho}{\sqrt{2|\Lambda_n|\sigma^2\log\log|\Lambda_n|}}
+\frac{|\Lambda_n|\rho -\mathbb E(\bar N_n)}{\sqrt{2|\Lambda_n|\sigma^2\log\log|\Lambda_n|}}
+\frac{\bar N_n- N_n}{\sqrt{2|\Lambda_n|\sigma^2\log\log|\Lambda_n|}}.
\]
From Theorem~\ref{main}, 
\[\limsup_n \frac{ N_n-|\Lambda_n|\rho}{\sqrt{2|\Lambda_n|\sigma^2\log\log|\Lambda_n|}}=1, \hspace{0.5cm} \text{almost surely,}\]
and by Proposition \ref{prop:mean-dev}, 
\[\lim_{n\to\infty} \frac{|\Lambda_n|\rho -\mathbb E(\bar N_n)}{\sqrt{2|\Lambda_n|\sigma^2\log\log|\Lambda_n|}}=0.\]
Therefore, for the LIL it is enough to prove that, with probability 1, there exists $K$ such that for any $n\ge K$ we have $|N_n-\bar N_n|\le \sqrt{|\Lambda_n|}$ (in our case $|\Lambda_n|=2n+1$). Consider the construction of $X$ and $X_{\Lambda_n}$ using the same field  $U$. We have
\[
\mathbb P(|N_n-\bar N_n|>M)\le \mathbb P\left(\sum_{i\in\Lambda_n}\1_{\{X(i)(U)\ne X_{\Lambda_n}(i)(U)\}}> M\right)
\le \mathbb P\left(T_L+T_R> M\right)
\]
where $T_L=\inf\{k\ge 0:U_{-n+k}<U_{-n+k-1}\wedge U_{-n+k+1}\}$ and $T_R=\inf\{k\ge0:U_{n-k}<U_{n-k-1}\wedge U_{n-k+1}\}$. Note that the last inequality comes from the dimensionality, because the presence of local minima shields the  boundary effects. Thus
\[
\mathbb P(|N_n-\bar N_n|>M)\le  2\mathbb P\left(T_L> M/2\right)
\]
by symmetry. But $\mathbb P\left(T_L> i\right)=\mathbb P\left(U_n>\ldots>U_{n-i-1}\right)=\frac{1}{(i+2)!}$, thus
\begin{equation}\label{eq:BC}
\mathbb P(|N_n-\bar N_n|>M)\le \frac{2}{\lceil M/2+2\rceil!}.     
\end{equation}
We therefore get
\[\sum_{n}\mathbb P(|N_n-\bar N_n|> \sqrt{|\Lambda_n|})<\infty,\]
which by Borel-Cantelli proves that for only finitely many $n$ we will have $|N_n-\bar N_n|> \sqrt{|\Lambda_n|}$, concluding the proof of the LIL.\\

For the CLT, it is enough to prove, by virtue of Slutsky's Theorem, that 
\[\left|\frac{ N_n-|\Lambda_n|\rho}{\sqrt{|\Lambda_n|\sigma^2}}-\frac{\bar N_n-\mathbb E(\bar N_n)}{\sqrt{|\Lambda_n|\sigma^2}}\right| \to 0, \hspace{0.3cm} \text{in probability.}\]
In fact, let $\epsilon>0$, by Markov's inequality,
\begin{equation}\label{conver_prob}    
\mathbb{P}\left[\left|
\frac{ N_n-|\Lambda_n|\rho}{\sqrt{|\Lambda_n|\sigma^2}}-\frac{\bar N_n-\mathbb E(\bar N_n)}{\sqrt{|\Lambda_n|\sigma^2}}\right|>\epsilon\right]\leq 
\frac{|\mathbb E(\bar N_n) - \rho|\Lambda_n||}{\epsilon\sqrt{|\Lambda_n|\sigma^2}}
+\frac{\mathbb E(|\bar N_n - N_n|)}{\epsilon\sqrt{|\Lambda_n|\sigma^2}}.
\end{equation}
By Proposition \ref{prop:mean-dev}, the first term on the right side of the inequality (\ref{conver_prob}) converges to zero. On the other hand, note that for all $n\geq 1$,
\begin{align*}
\mathbb E(|\bar N_n - N_n|)
&\leq \sum_{i\in \Lambda_n}\mathbb P(X(i)(U)\ne X_{\Lambda_n}(i)(U))\\
&\leq \sum_{i\in \Lambda_n}\mathbb P(\mathcal{A}(\{i\})\not\subset \Lambda_n)\\
&\leq \sum_{i=-n}^n \left[\frac{1}{(n-i+2)!}+\frac{1}{(i+n+2)!}\right]<K<\infty,
\end{align*}
where $K$ is some positive number. Therefore, the second term on the right side of the inequality (\ref{conver_prob}) also converges to zero. In this way, we obtain the desired convergence in probability, concluding the proof of the CLT.\\

We now prove the last part of the proposition. Observe that 
\begin{align*}
\mathbb P\left(\left|{\bar N_n}-\rho|\Lambda_n|\right|>\epsilon\right)&\le\mathbb P\left(\left|{\bar N_n}- N_n\right|+\left|N_n-\rho|\Lambda_n|\right|>\epsilon\right)\\
&\le\mathbb P\left(\left|{\bar N_n}- N_n\right|>\frac{\epsilon}{2}\right)+\mathbb P\left(\left|N_n-\rho|\Lambda_n|\right|>\frac{\epsilon}{2}\right)\\
&\le \frac{2}{\lceil \epsilon/4+2\rceil!}+e^{\frac{1}{e}-\frac{\epsilon^2}{16eB|\Lambda_n|}}.
\end{align*}
where the last line has been obtained by using (\ref{eq:BC}) and (\ref{eq:concentrationmain}).
 \end{proof}

\section*{Acknowledgments} This research was supported by FAPESP (grant number 2022/08948-2) and Universidad de Antioquia.

\bibliographystyle{plain}
\bibliography{Referencias}
\end{document}